\documentclass[reqno,10pt,centertags,draft]{amsart}
\usepackage{amsmath,amsthm,amscd,amssymb,latexsym,esint,upref,stmaryrd,
enumerate,verbatim,yfonts,bm,mathtools}
\usepackage{color}
\usepackage{hyperref} 
\newcommand*{\mailto}[1]{\href{mailto:#1}{\nolinkurl{#1}}}
\newcommand{\arxiv}[1]{\href{http://arxiv.org/abs/#1}{arXiv:#1}}

\newcommand{\N}{{\mathbb N}}

\newcommand{\bbC}{{\mathbb{C}}}

\newcommand{\bbN}{{\mathbb{N}}}

\newcommand{\bbR}{{\mathbb{R}}}
\newcommand{\bbS}{{\mathbb{S}}}

\newcommand{\cB}{{\mathcal B}}
\newcommand{\cC}{{\mathcal C}}
\newcommand{\cD}{{\mathcal D}}

\newcommand{\cF}{{\mathcal F}}

\newcommand{\cH}{{\mathcal H}}

\newcommand{\cK}{{\mathcal K}}

\newcommand{\beq}{\begin{equation}}
\newcommand{\enq}{\end{equation}}

\DeclareMathOperator{\ran}{ran}
\DeclareMathOperator{\dom}{dom}

\DeclareMathOperator*{\slim}{s-lim}

\renewcommand{\Re}{\text{\rm Re}}
\renewcommand{\Im}{\text{\rm Im}}

\newcommand{\no}{\notag}
\newcommand{\lb}{\label}
\newcommand{\f}{\frac}
\newcommand{\ul}{\underline}
\newcommand{\ol}{\overline}

\newcommand{\hatt}{\widehat} 
\newcommand{\bi}{\bibitem}

\newcommand{\ma}[2]{\left(\begin{array}{#1} #2 \end{array} \right)}

\newcommand{\pa}{\partial}

\newcommand{\LN}{\big[L^2(\bbR^n)\big]^N} 
\newcommand{\LNs}{[L^2(\bbR^n)]^N} 
 
\newcommand{\LoneN}{\big[L_1^2(\bbR^n)\big]^N} 
\newcommand{\LoneNs}{[L_1^2(\bbR^n)]^N} 
 
\newcommand{\LmoneNs}{[L_{-1}^2(\bbR^n)]^N} 
\newcommand{\LonehN}{\big[L_{1/2}^2(\bbR^n)\big]^N} 
\newcommand{\LonehNs}{[L_{1/2}^2(\bbR^n)]^N} 
 
\newcommand{\WoneN}{\big[W^{1,2}(\bbR^n)\big]^N} 
 
\newcommand{\WtwoN}{\big[W^{2,2}(\bbR^n)\big]^N} 
 
\newcommand{\SN}{[S(\bbR^n)]^N} 
 
\let\geq\geqslant
\let\leq\leqslant

\makeatletter
\def\theequation{\@arabic\c@equation}

\allowdisplaybreaks 
\numberwithin{equation}{section}

\newtheorem{theorem}{Theorem}[section]
\newtheorem{proposition}[theorem]{Proposition}
\newtheorem{lemma}[theorem]{Lemma}

\newtheorem{hypothesis}[theorem]{Hypothesis}

\theoremstyle{remark}
\newtheorem{remark}[theorem]{Remark}

\begin{document}

\title[A Global Limiting Absorption Principle]{On the Global Limiting Absorption Principle \\ 
for Massless Dirac Operators}

\author[A.\ Carey et al.]{Alan Carey}  
\address{Mathematical Sciences Institute, Australian National University, 
Kingsley St., Canberra, ACT 0200, Australia 
and School of Mathematics and Applied Statistics, University of Wollongong, NSW, Australia,  2522}  
\email{\mailto{acarey@maths.anu.edu.au}}
\urladdr{\url{http://maths.anu.edu.au/~acarey/}}
  
\author[]{Fritz Gesztesy}
\address{Department of Mathematics, 
Baylor University, One Bear Place \#97328,
Waco, TX 76798-7328, USA}
\email{\mailto{Fritz\_Gesztesy@baylor.edu}}
\urladdr{\url{http://www.baylor.edu/math/index.php?id=935340}}

\author[]{Jens Kaad}
\address{Department of Mathematics and Computer Science,
University of Southern Denmark, Campusvej 55, DK-5230 Odense M, Denmark}
\email{\mailto{kaad@imada.sdu.dk}}

\author[]{Galina Levitina} 
\address{School of Mathematics and Statistics, UNSW, Kensington, NSW 2052,
Australia} 
\email{\mailto{g.levitina@student.unsw.edu.au}}

\author[]{Roger Nichols}
\address{Mathematics Department, The University of Tennessee at Chattanooga, 
415 EMCS Building, Dept. 6956, 615 McCallie Ave, Chattanooga, TN 37403, USA}
\email{\mailto{Roger-Nichols@utc.edu}}
\urladdr{\url{http://www.utc.edu/faculty/roger-nichols/index.php}}

\author[]{Denis Potapov}
\address{School of Mathematics and Statistics, UNSW, Kensington, NSW 2052,
Australia} 
\email{\mailto{d.potapov@unsw.edu.au}}

\author[]{Fedor Sukochev}
\address{School of Mathematics and Statistics, UNSW, Kensington, NSW 2052,
Australia} 
\email{\mailto{f.sukochev@unsw.edu.au}}

\date{\today}
\thanks{A.C., G.L., and F.S. gratefully acknowledge financial support from the Australian Research Council.  J.K. is supported by the DFF-research project 2 ``Automorphisms and invariants of operator algebras,'' no. 7014-00145B and by the Villum foundation (Grant 7423).}


\begin{abstract} 
We prove a global limiting absorption principle on the entire real line for free, massless Dirac operators $H_0 = \alpha \cdot (-i \nabla)$ for all space dimensions $n \in \bbN$, $n \geq 2$. This is a new result for all dimensions other than three, in particular, it applies to the two-dimensional case which is known to be of some relevance in applications to graphene. 

We also prove an essential self-adjointness result for first-order matrix-valued differential operators with Lipschitz coefficients.  
\end{abstract}

\maketitle



\section{Introduction} \lb{s1}

This paper was motivated by recent investigations of the Witten index (a possible substitute for the Fredholm index) for classes of non-Fredholm operators, a prime example of which being the massless Dirac operator, see, for instance, \cite{CGLPSZ16}--\cite{CGPST17}. 

The first result on a global limiting absorption principle on the entire real line for free, massive Dirac 
operators $H_0(m) = \alpha \cdot (-i \nabla) + m \, \beta$ (cf.\ \eqref{3.3a} for details), in the case of 
dimensions $n=3$, is due to Iftimovici and M{\u a}ntoiu \cite{IM99} in 1999. The first such result 
for free, massless Dirac operators $H_0 = \alpha \cdot (-i \nabla)$ (cf. \eqref{2.2}), again in dimensions $n=3$, was proved by Sait{\= o} and Umeda \cite{SU08} in 2008 upon relying on quite different methods as part of their study of zero-energy eigenvalues and resonances. Since, apparently, no other results are available in the massless case, we now fill this gap, and by suitably modifying the approach of Iftimovici and M{\u a}ntoiu, we prove a global limiting absorption principle on $\bbR^n$ for free,  massless Dirac operators in all dimensions $n \in \bbN$, $n \geq 2$. 

It is gratifying that the new results in dimensions $n \in \bbN \backslash \{3\}$, $n \geq 2$, include, 
in particular, the two-dimensional case $n=2$ which is known to have some relevance in applications to graphene.

We emphasize that proving a limiting absorption principle for free (usually, constant coefficient) operators $H_0$ (resp., $H_0(m)$) is always a first step in proving similar statements (typically, 
away from essential spectrum thresholds, though) for interacting Hamiltonians $H = H_0 + V$ 
(resp., $H(m) = H_0(m) + V$) on the basis of a sophisticated perturbative approach. The 
strategy behind such an approach has been spelled out in great detail, for instance, by Yafaev 
in \cite[Sect.~4.6, 4.7]{Ya92}.  

While we follow the broad contours outlined in the approach employed by Iftimovici and 
M{\u a}ntoiu \cite{IM99} in the massive case, $m > 0$, there are notable differences in our 
treatment of the massless case, $m = 0$; of these we note, in particular, the following: 

$\bullet$ First, we do not rely on Nelson's commutator theorem in proving essential selfa-djointness
of the conjugate operator $A$ introduced in \eqref{2.6a}. Instead, we employ an extension of an essential self-adjointness result for general first-order matrix-valued differential operators (such as, $A$, upon applying the Fourier transform) going back to Chernoff, \cite{Ch73}, but see also \cite{HR00} for a more modern treatment. Here, we treat only the flat case but permit Lipschitz coefficients instead of the more restrictive (and more common) smoothness assumption. We expect this topic, to be treated in detail in Section \ref{s2}, to be of independent interest.


$\bullet$ Second, given our focus on massless Dirac operators in all dimensions $n \geq 2$, we were led to a new 
conjugate operator $A$ in \eqref{2.6a} when compared to the massive case discussed in 
\cite{IM99}. This causes a variety of additional technical difficulties in Section \ref{s3} briefly 
outlined in the paragraph following \eqref{2.7a}. 

$\bullet$ Third, rather than applying Hardy's inequality as in \cite{IM99}, which is only 
applicable in dimensions $n \geq 3$ (a restriction we wanted to avoid by all means), we now 
apply Kato's inequality (cf.~\eqref{2.20a}, and especially, \eqref{2.22a}). 

Our principal result, the global limiting absorption principle on the real line for $H_0$ is presented in 
Section \ref{s3} (see Theorem \ref{t3.15}). This is followed by a standard application to scattering theory for the pair of self-adjoint operators $(H=H_0+V, H_0)$, with sufficiently weak interactions 
$V$ (cf.\ Theorem \ref{t3.17}), and some remarks that put our results in proper perspective.

We conclude this introduction by briefly summarizing some of the notation used throughout 
this paper: Vectors in $\bbR^n$ are denoted by $ x = (x_1,\ldots,x_n)\in \bbR^n$ or $ p  = (p_1,\ldots,p_n)\in \bbR^n$, $n \in \bbN$. For $ x = (x_1,\ldots,x_n)\in \bbR^n$ we abbreviate 
\begin{equation}
\langle x \rangle = (1+| x |^2)^{1/2},
\end{equation}
where $|x|= \big(x_1^2+\cdots+x_n^2\big)^{1/2}$ denotes the standard Euclidean norm of $x \in \bbR^n$, $n \in \bbN$.

The dot symbol, ``$\, \cdot \,$'',\ is used in three different ways: First, it denotes the standard scalar product in $\bbR^n$, 
\begin{equation}
x \cdot y = \sum_{j=1}^n x_j y_j, \quad x =(x_1,\dots,x_n), \, y = (y_1,\dots,y_n) \in \bbR^n.
\end{equation}
Second, we will also use it for $n$-vectors of operators, ${\ul A} = (A_1,\ldots,A_n)$ 
and ${\ul B} = (B_1,\ldots, B_n)$ acting in the same Hilbert space in the form 
\begin{equation}
{\ul A} \cdot {\ul B} = \sum_{j=1}^n A_j B_j,
\end{equation}
whenever it is obvious how to resolve the domain issues of the possibly unbounded operators involved. 
Moreover, for $T$ an operator in some Hilbert space $\cH$ and 
$A = (a_{j,k})_{1 \leq j,k \leq N} \in \bbC^{N \times N}$ an $N \times N$ matrix with constant 
complex-valued entries acting in $\bbC^N$, $N \in \bbN$, we will avoid tensor product notation in
\begin{equation}
T \otimes A \, \text{ in } \, \cH \otimes \bbC^N 
\end{equation} 
such that 
\begin{equation}
\cH \otimes \bbC^N \, \text{ is identified with } \, 
\cH^N = \begin{pmatrix} \cH \\ \vdots \\ \cH \end{pmatrix}, 
\end{equation}
and 
\begin{equation}
T \otimes A \, \text{ is identified with } \, T A = (T a_{j,k})_{1 \leq j,k \leq N} = 
(a_{j,k} T)_{1 \leq j,k \leq N} = A T.    \lb{1.5} 
\end{equation} 
That is, we interpret $T \otimes A$ as entrywise multiplication, resulting in an $N \times N$ block operator matrix $TA = AT$. Thus, if ${\ul T} = (T_1,\dots,T_n)$, with $T_j$, $1 \leq j \leq n$, operators 
in $\cH$, and ${\ul A} = (A_1,\dots,A_n)$, with $A_j \in \bbC^{N \times N}$, $1 \leq j \leq n$, $N \times N$ matrices in $\bbC^N$, we will finally employ the dot symbol also in the form
\begin{equation}
{\ul T} \cdot {\ul A} = \sum_{j=1}^n T_j A_j = \sum_{j=1}^n A_j T_j = {\ul A} \cdot {\ul T},
\end{equation} 
where $T_j A_j = A_j T_j$, $1 \leq j \leq n$, are defined as in \eqref{1.5}. In the interest of clarity we temporarily underlined vectors of operators (and matrices) such as ${\ul T} = (T_1,\dots,T_n)$; we will refrain from doing so in the bulk of this manuscript. 

For $X$ a given set, $A \in X^{N \times N}$, $N \in \bbN$, represents an $N \times N$ matrix $A$ with entries in $X$.  

Let $\cH$, $\cK$ be separable complex Hilbert spaces, 
$(\, \cdot \,,\, \cdot \,)_{\cH}$ the scalar 
product in $\cH$ (linear in the second argument), 
$\|\, \cdot \, \|_{\cH}$ the norm on $\cH$, 
and $I_{\cH}$ the identity operator in $\cH$.
If $T$ is a linear operator mapping (a subspace of) a Hilbert space 
into another, then $\dom(T)$ and $\ker(T)$ denote the domain 
and kernel (i.e., null space) of $T$. The closure of a closable operator 
$A$ is denoted by $\ol A$. 

The resolvent set and spectrum of a closed operator $T$ are denoted by $\rho(T)$ and $\sigma(T)$, respectively. 

The Banach space of bounded linear operators on a separable 
complex Hilbert space $\cH$ is denoted by $\cB(\cH)$. 

For a densely defined closed operator $S$ in $\cH$ we employ the abbreviation $\langle S \rangle := \big(I_{\cH} + |S|^2\big)^{1/2}$, and 
similarly, if $T = (T_1,\dots,T_n)$, with $T_j$ densely defined and closed 
in $\cH$, $1 \leq j \leq n$, 
\begin{equation}
\langle T \rangle = (I_{\cH} + | T |^2)^{1/2}, \quad | T | 
= \big(|T_1|^2+\cdots+|T_n|^2\big)^{1/2},
\end{equation}
whenever it is obvious how to define $|T_1|^2+\cdots+|T_n|^2$ as a nonnegative self-adjoint operator. 

To simplify notation, we will frequently omit Lebesgue measure whenever possible 
and simply use $L^2(\bbR^n)$ instead of 
$L^2(\bbR^n; d^nx)$, etc. 
If no confusion can arise, the identity operator in $L^2(\bbR^n)$   
is simply denoted by $I$, and $I_N$ represents the identity operator in $\bbC^N$, 
$N \in \bbN$.  

The symbol $\cF$ is used to denote the Fourier transform and $\widehat f : = \cF f$.  For $N\in \mathbb{N}$, the Fourier transform of functions in $\LN$is taken component-wise, that is, if $f=(f_1,\ldots,f_N)^{\top}\in \LN$, then $\widehat{f} = (\widehat{f}_1,\ldots,\widehat{f}_N)^{\top}$.  

If $\Omega\subseteq \bbR^n$ is open, then $C_0^{\infty}(\Omega)$ denotes the set of infinitely differentiable functions on $\bbR^n$ with compact support in $\Omega$.  In addition, $S(\bbR^n)$ denotes the Schwartz space of rapidly decreasing functions on $\bbR^n$, $S'(\bbR^n)$ the space of tempered distributions, and $W^{k,p}(\bbR^n)$, $k \in \bbN$, $p \geq 1$, the standard $L^p$-based Sobolev spaces.

We abbreviate $\bbC_{\pm} = \{z \in \bbC \,|\, \pm \Im(z) > 0\}$. The symbol 
$\lfloor \, \cdot \, \rfloor$ denotes the floor function on $\bbR$, that is, $\lfloor x \rfloor$ characterizes the largest integer less than or equal to $x \in \bbR$.  

Following a standard practice in Mathematical Physics, we will simplify the notation of operators of multiplication by a scalar or matrix-valued function $V$ and hence use $V$ rather than the more elaborate symbol $M_V$ throughout this manuscript.

\section{Essential Self-Adjointness of First-Order Differential Operators With 
Lipschitz Coefficients} \lb{s2}

We start with a general self-adjointness result for first-order differential expressions which we 
believe is of independent interest. It will be applied in the subsequent section in connection 
with the operator $A$ in \eqref{2.6a}.

Let $n, N \in \mathbb{N}$ be fixed.  We consider $N\times N$ block operator matrices of bounded operators of multiplication 
\begin{equation}
F_j \in \cB\big(L^2(\mathbb{R}^n)\big)^{N \times N}, \quad j\in \bbN,\, 1\leq j\leq n,   \lb{2.1a}
\end{equation}
(e.g., $F_j$ can be $N \times N$ matrix-valued operators of multiplication satisfying $F_j = (F_{j,k,\ell})_{1 \leq k, \ell \leq N}$, $F_{j,k,\ell} \in L^{\infty}(\bbR^n)$, $1 \leq j \leq n$, 
$1 \leq k, \ell \leq N$) together with the unbounded, closable operators 
\begin{equation}
-i \pa_j : \WoneN \to \LN,  \quad j\in \bbN, \, 1\leq j\leq n.
\end{equation}
where $\partial_j = \partial/\partial x_j$, $1 \leq j \leq n$. We put
\begin{equation}
\nabla := \ma{c}{\partial_1 \\ \vdots \\ \partial_n} : \WoneN
\to \Big(\LN\Big)^{n} = \begin{pmatrix} \LN \\ \vdots \\ \LN \end{pmatrix},
\end{equation}
where $\Big(\LN\Big)^{n}$ denotes the $n$-fold direct sum of $\LN$ with itself.  
One notes that $\nabla$ is, in fact, densely defined and closed. Define
\begin{equation}
-\Delta := \nabla^* \nabla = - \sum_{j = 1}^n \pa_j^2 : \WtwoN \to \LN.
\end{equation}

The main goal of this section is to prove the essential self-adjointness of a class of first-order differential operators.  To make the result precise, we introduce the following set of assumptions, which will remain in effect throughout the remainder of this section.

\begin{hypothesis}\lb{h2.1}
$(i)$ Suppose $F_j$, $1 \leq j \leq n$ satisfy \eqref{2.1a}. \\[1mm]  
$(ii)$ Let $\mathcal{D} \subseteq \LN$ be a core for $\nabla$ with the property that
\begin{equation}
F_j( \mathcal{D}) \subseteq \mathcal{D}, \quad F_j^*( \mathcal{D}) \subseteq \mathcal{D},\quad j\in \bbN,\, 1\leq j\leq n.
\end{equation}
$(iii)$ Suppose that the commutators
\begin{equation}
\big[ -i \partial _k , F_j \big], \, \big[-i \pa_k, F_j^*\big] : \mathcal{D} \to \LN,\quad j,k\in \bbN,\, 1\leq j,k\leq n,
\end{equation}
extend to bounded operators 
\begin{equation}
d_k(F_j) , \, d_k(F_j^*)\in \cB\Big(\LN\Big),\quad j,k\in \bbN,\,  1\leq j,k\leq n,
\end{equation}
respectively.\\[1mm]
$(iv)$ Define the symmetric unbounded first-order differential operator
\begin{equation}
L := -i \sum_{j = 1}^n F_j \partial_j - i \sum_{j = 1}^n \partial_j F_j^* : \mathcal{D} \to \LN.    \lb{2.9} 
\end{equation}
\end{hypothesis}

\begin{proposition}\label{p2.2}
Assume Hypothesis \ref{h2.1}.  Then items $(ii)$ and $(iii)$ in Hypothesis \ref{h2.1} hold with $\cD$ replaced by $\WoneN$.  Moreover,
\begin{equation}\lb{a.10s}
[ W^{1,2}(\bbR^n) ]^N \subseteq \dom \big( \overline{L}\big).
\end{equation}
\end{proposition}
\begin{proof}
The proof that items $(ii)$ and $(iii)$ in Hypothesis \ref{h2.1} hold with the choice $\cD = \WoneN$ is standard (cf., e.g., \cite[Proposition 2.1]{FMR14}).

One notices that $L$ is symmetric by construction.  In particular, $L$ is closable and hence 
$\overline{L}$ is well-defined. To show that $[ W^{1,2}(\bbR^n) ]^N \subseteq \dom \big( \overline{L}\big)$, define the bounded operators
\begin{equation}
F = (F_1, \dots, F_n) : \big(\LN\big)^{n} \to \LN, 
\end{equation}
and
\begin{equation}
(F^*)^{\top} = (F_1^*, \dots, F_n^*) : \big(\LN\big)^{n} \to \LN .
\end{equation}
Using matrix multiplication one verifies that
\begin{equation}
L = F \cdot (-i \nabla) + (F^*)^{\top} \cdot (-i \nabla) + \sum_{j = 1}^n d_j(F_j^*) : \mathcal{D} \to \LN.
\end{equation}
Since $F,(F^*)^{\top}$, and $d_j(F_j^*)$, $1\leq j\leq n$, are all bounded and $\mathcal{D}$ is a core for $\nabla$, one infers
\begin{equation}
\WoneN  = \dom( \nabla ) \subseteq \dom( \overline{L} ),
\end{equation}
establishing \eqref{a.10s}.
\end{proof}

The proof of the main result of this section relies on the following lemma which yields a useful resolvent-type identity.

\begin{lemma}\label{l2.3}
Assume Hypothesis \ref{h2.1} and set
\begin{equation}
R_m := \big(I_{\LNs} - m^{-1} \Delta\big)^{-1} : \LN \to \LN, \quad m \in \N.
\end{equation}
Then the following resolvent identity holds:
\begin{equation}\label{eq:commutator}
\begin{split}
\big( \overline{L} R_m - R_m L^* \big) g 
& = - \frac{1}{m} \sum_{j,k = 1}^n \big(\pa_k R_m d_k(F_j) + R_m d_k(F_j) \pa_k \big) \pa_j 
R_m g \\
& \quad - \frac{1}{m} \sum_{j,k = 1}^n \pa_j \big(\pa_k R_m d_k(F_j^*) + R_m d_k(F_j^*) \pa_k\big) 
R_m g 
\end{split}
\end{equation}
for all $g \in \dom(L^*)$ and all $m \in \N$. In particular, the commutator
\begin{equation}
\overline{L} \big(I_{\LNs} - m^{-1} \Delta\big)^{-1} 
- \big(I_{\LNs} - m^{-1} \Delta\big)^{-1} L^* : \dom(L^*) \to \LN
\end{equation}
extends to a bounded operator $X_m : \LN \to \LN$ for all $m \in \N$ and
\begin{equation}
\sup_{m \in \N} \| X_m \|_{\cB(\LNs)} < \infty .
\end{equation}
\end{lemma}
\begin{proof}
The second claim of the lemma follows from the estimate
\begin{equation}
\bigg\| \pa_j \pa_k R_m \bigg\|_{\cB( \LNs)} \leq m,\quad m \in \N,
\end{equation}
and from our assumption that
\begin{equation}
d_k(F_j^*),d_k(F_j) : \LN \to \LN,\quad 1\leq j,k \leq n,\, j,k\in \bbN,
\end{equation}
are bounded. One notes also that
\begin{equation}\lb{2.20p}
X_m^* = - X_m, \quad m \in \N.
\end{equation}
Therefore, we focus on proving identity \eqref{eq:commutator}. 

We first prove identity \eqref{eq:commutator} for $f \in \dom(L) = \cD$. To this end, let 
$f \in \cD$ be given. The desired identity is a result of the following computation (applying Proposition \ref{p2.2} to take care of domain issues):
\begin{equation}
\begin{split}
\overline{L} R_m f - R_m L f
& = -i \sum_{j = 1}^n (F_j R_m - R_m F_j) \pa_j f 
-i \sum_{j = 1}^n \pa_j ( F_j^* R_m - R_m F_j^* ) f    \\
& = \frac{i}{m} \sum_{j = 1}^n ( \Delta R_m F_j R_m - R_m F_j \Delta R_m) (\pa_j f) \\
& \quad + \frac{i}{m} \sum_{j = 1}^n \pa_j ( \Delta R_m F_j^* R_m - R_m F_j^* \Delta R_m) f \\
& = - \frac{1}{m} \sum_{j,k = 1}^n \big(\pa_k R_m d_k(F_j) R_m + R_m d_k(F_j) \pa_k R_m\big) 
(\pa_j f)    \\
& \quad - \frac{1}{m} \sum_{j,k = 1}^n \pa_j \big(\pa_k R_m d_k(F_j^*) R_m + R_m d_k(F_j^*) 
\pa_k R_m\big) f.
\end{split}
\end{equation} 

We have now shown that identity \eqref{eq:commutator} holds for all $f \in \dom(L)$, that is, 
\begin{equation}
\big( \overline{L} R_m - R_m L \big) f = X_m f, \quad f \in \dom(L).
\end{equation}
For $g \in \dom(L^*)$ one then computes
\begin{align}
\big(\big( \overline{L} R_m - R_m L^* \big) g, f \big)_{\LNs}&= (g, (R_m L - \overline{L} R_m) f)_{\LNs}\no\\
&= - (g, X_m f)_{\LNs}\no\\
&= (X_m g, f)_{\LNs},\quad  f \in \dom(L),
\end{align}
where the last equality makes use of \eqref{2.20p}.  Since $\dom(L)$ is dense in $\LN$, this proves identity \eqref{eq:commutator}.
\end{proof}

Lemma \ref{l2.3}, in particular \eqref{eq:commutator}, is essentially a restatement of the resolvent identity. The slightly complicated form of this identity is due to our assumptions on the coefficients $F_j, F_j^*$, $1\leq j\leq n$. Indeed, we do \emph{not} assume that these operators preserve the domain of the Laplacian $\Delta : \WtwoN \to \LN$.  In fact, they are only assumed to preserve the domain of the gradient 
$\nabla : \WoneN \to \Big(\LN\Big)^{n}$.

With these preparations in place, we are now ready to state and prove the main result of this section.

\begin{theorem}\label{t2.4}
Assume Hypothesis \ref{h2.1}. Then $L : \cD \to [ L^2(\bbR^n) ]^N$ is essentially self-adjoint.
\end{theorem}
\begin{proof}
The operator $L$ is symmetric by construction.  To prove that $L$ is essentially self-adjoint it suffices to show that
\begin{equation}
\dom(L^*) \subseteq \dom( \overline{L}).
\end{equation}
To this end, let $g \in \dom(L^*)$. One notes that
\begin{equation}
\lim_{m\to \infty}\big(I_{\LNs} - m^{-1}\Delta\big)^{-1} g = g,
\end{equation}
and, moreover, 
\begin{equation}
\big(I_{\LNs} - m^{-1} \Delta\big)^{-1} g \in \dom( \overline{L}),
\end{equation}
since 
$\WtwoN = \dom(-\Delta) \subseteq \WoneN \subseteq \dom(\overline{L})$. 
By Lemma \ref{l2.3} one infers
\begin{equation}
\overline{L} \big(I_{\LNs} - m^{-1} \Delta\big)^{-1} g
= \big(I_{\LNs} - m^{-1} \Delta\big)^{-1} L^* g + X_m g, \quad m \in \N,
\end{equation}
so that the sequence
\begin{equation}
\Big\{ \overline{L} \big(I_{\LNs} - m^{-1} \Delta\big)^{-1} g \Big\}_{m\in \bbN}
\end{equation}
is bounded in the norm of $\LN$, implying $g \in \dom(\overline{L})$. 
\end{proof}

\begin{remark}\lb{r2.5}
Theorem \ref{t2.4} extends Proposition~10.2.11 in Higson and Roe \cite{HR00} (see also \cite{Ch73}) in the smooth context,
while we now permit Lipschitz coefficients $F_j$ in the differential expression $L$ in \eqref{2.9}. More precisely, assuming smooth coefficients, Higson and Roe prove that any first-order, symmetric differential operator of finite propagation speed acting on the smooth compactly supported sections of a smooth hermitian vector bundle on any complete manifold (without boundary) is essentially self-adjoint. While the results proved here are very similar in nature (although restricted to the flat case), the actual strategy of proof in Theorem \ref{t2.4} differs from the one employed in \cite{HR00} and is inspired by ideas appearing in \cite{Ka17} and \cite{KaLe12}, \cite{MeRe16} in the context of Hilbert $C^*$-modules. In a broader picture our methods relate to the noncommutative geometry program \cite{Co94}, since our proof is in some sense coordinate-free and can therefore be readily generalized to a much wider array of unbounded operators of the form $\sum_j A_j \cdot D_j + B$ (with $D_j$ and $D_k$ commuting for all 
$1 \leq j, k \leq n$), provided that an appropriate reference operator, for example $\sum_j D_j^2$, 
is already well-understood. 

For a very recent approach to essential self-adjointness of first-order differential operators with applications to Dirac-type operators we refer to \cite{BS17} (the approach in \cite{BS17} is quite different, relying on ellipticity conditions which are not used in our setup). 

Needless to stress, self-adjointness is one of the single most important properties of an unbounded operator, because of its implications to the spectrum and to the Borel functional calculus. 
\hfill $\diamond$
\end{remark}

\section{A Global Limiting Absorption Principle for \\ Free, Massless Dirac Operators} 
\lb{s3}

In 1999, Iftimovici and M{\u a}ntoiu \cite{IM99} proved a global limiting absorption principle, 
that is, one on the entire real axis and hence including threshold energies $\pm m$ for the free, massive Dirac operator $H_0(m)$ with mass $m > 0$ in three dimensions. The first proof of a global limiting absorption principle for massless Dirac operators $H_0$ in three dimensions is due to 
Sait{\= o} and Umeda \cite{SU08} in 2008. As no other result on a global limiting absorption principle 
in the massless case is known to us, we now fill this gap and upon modifying the approach by Iftimovici and M{\u a}ntoiu for $m > 0$, we treat free, massless ($m=0$) Dirac operators 
$H_0$ in all dimensions $n \in \bbN$, $n \geq 2$. This includes, in particular, a new result for the case $n=2$ 
which is known to be connected to applications to graphene.  

Here the notion ``free'' Dirac operator refers to a particular constant coefficient first-order matrix-valued differential operator with vanishing electric (and magnetic) potentials, see \eqref{2.2}. 

To rigorously define the free massless $n$-dimensional Dirac operators to be studied in this manuscript, we introduce the following basic assumption.

\begin{hypothesis} \lb{h3.1}  
Let $n \in \bbN$, $n\geq 2$, set $N=2^{\lfloor(n+1)/2\rfloor}$, and denote by 
$\alpha_j$, $1\leq j\leq n$, $\alpha_{n+1} := \beta$,  $n+1$ anti-commuting self-adjoint $N\times N$ matrices with squares equal to $I_N$, that is, 
\begin{equation}\lb{2.1}
\alpha_j^*=\alpha_j,\quad \alpha_j\alpha_k + \alpha_k\alpha_j = 2\delta_{j,k}I_N, 
\quad 1\leq j,k\leq n+1.
\end{equation}
\end{hypothesis}

Given Hypothesis \ref{h3.1}, we introduce in $\LN$ the free massless Dirac operator as follows, 
\begin{equation}
H_0 = \alpha \cdot (-i \nabla) = \sum_{j=1}^n \alpha_j (-i \partial_j),\quad \dom(H_0) = \WoneN.  
\lb{2.2}
\end{equation} 

Employing the relations \eqref{2.1}, one observes that 
\begin{equation} 
H_0^2 = - I_N \Delta, \quad \dom\big(H_0^2\big) = \WtwoN.   \lb{2.7} 
\end{equation}

For completeness we also recall that the massive free Dirac operator in $\LN$ associated 
with the mass parameter $m > 0$ then would be of the form
\begin{equation}
H_0(m) = H_0 + m \, \beta, \quad \dom(H_0(m)) = \WoneN, \; m > 0, \; \beta = \alpha_{n+1}, 
\lb{3.3a} 
\end{equation}
but we will primarily study the massless case $m=0$ in this paper.

In the special one-dimensional case $n=1$, one can choose $\alpha_1$ to be one of the three Pauli matrices. Similarly, in the massive case, $\beta$ would typically be a second Pauli matrix (different from $\alpha_1$).

The main goal of this section is to obtain a uniform limiting absorption principle for the free $n$-dimensional massless Dirac operator in dimensions $n \geq 2$.  The method of proof employed relies on Kato's inequality (cf.\ \eqref{2.20a}, not to be confused with his distributional inequality) and the construction of an auxiliary operator $A$ that has a positive commutator with the free massless Dirac operator $H_0$. 

To set the stage for the definition of $A$ we introduce the following assumption:

\begin{hypothesis} \lb{h3.2} Let $\eta:\bbR^n\to \bbR$ denote a radial function of the form
\begin{equation}
\eta(p) = h(|p|),\quad p\in \bbR^n,
\end{equation}
where $h : [0,\infty) \to [0,1]$ is defined by
\begin{equation}\lb{3.4a} 
h (r) := 
\begin{cases}
r, & r \in [0,1/2), \\
k (r), & r \in [1/2,1),  \\
1, & r \in [1,\infty),
\end{cases}
\end{equation}
and the function $k : [1/2,1)\to [0,\infty)$ is nondecreasing and chosen so that 
$\eta \in C^{\infty}(\bbR^n\backslash\{0\})$. 
\end{hypothesis}

The operator of multiplication by the independent variable $x_j$ in $\LN$ will be denoted by $Q_j$ and we shall write
\begin{equation}
Q=(Q_1,\ldots,Q_n).
\end{equation}
Given Hypotheses \ref{h3.1} and \ref{h3.2}, we introduce in $\LN$ the operator 
\begin{align}
& A = \frac{1}{2} \Big[(\alpha \cdot (-i \nabla))(- \Delta)^{-1} \eta (-i \nabla)((-i \nabla) \cdot Q)   \no \\
& \hspace*{11mm} + (Q \cdot (-i \nabla)) (- \Delta)^{-1} \eta (-i \nabla)
(\alpha \cdot (-i \nabla)) \Big],      \lb{2.6a} \\ 
& \dom(A)=\cD_0 (\bbR^n),    \no
\end{align}
where
\begin{equation}
\cD_0(\bbR^n) = \cF^{-1} \Big([C_0^{\infty}(\bbR^n\backslash \{0\})]^N\Big) 
\subset \SN. \lb{2.10} 
\end{equation}  
That is, $\cD_0(\bbR^n)$ consists of functions whose Fourier transforms have compact support and no support in a neighborhood of $p=0$.

In addition, we introduce 
\begin{equation}
B = \eta (-i \nabla) \in \cB\Big(\LN\Big),    \lb{2.7a} 
\end{equation} 
defined via the spectral theorem. 

We emphasize that the necessity of including the factor 
$\eta (-i \nabla)$ in the definition \eqref{2.6a} of $A$ considerably complicates matters as at various 
occasions we will have to consider $B^{-1}$ in combination with other operators. The 
corresponding massive case, $m > 0$, as treated by Iftimovici and M{\u a}ntoiu \cite{IM99}, corresponds to the bounded operator $\big(- \Delta + m^2 I\big)^{-1}$ instead of $(- \Delta)^{-1} $ in $A$,  
and hence does not require the introduction of the term $\eta (-i \nabla)$ in \eqref{2.6a}. Naturally, this 
considerably influences some technical aspects in the proofs of this section.

For any $n \in \bbN$, we also introduce the scale of weighted $L^2$-spaces, 
\begin{equation}
L^2_s(\bbR^n) = \big\{ f \in L^2(\bbR^n)\, |\, \langle Q\rangle^s f \in L^2(\bbR^n) \big\},\quad L^2_{-s}(\bbR^n) = \big[L^2_s(\bbR^n)\big]^*,\quad s \geq 0.
\end{equation}

\begin{proposition}\lb{p3.3}
Assume Hypotheses \ref{h3.1} and \ref{h3.2}. ~Then the operator $A$ is essentially self-adjoint on 
$\cD_0(\bbR^n)$ and 
\begin{equation}\lb{2.11zzz}
\LoneN \subseteq \dom\big(\ol{A}\big).
\end{equation}
\end{proposition}
\begin{proof}
Considering $\hatt A = \cF A \cF^{-1}$, $\hatt A$ reads 
\begin{align}
\begin{split} 
& \hatt A = \frac{1}{2} \Big[(\alpha \cdot Q)| Q |^{-2}\eta( Q )(Q \cdot (-i \nabla)) 
+ ( (-i \nabla) \cdot Q) | Q |^{-2}\eta(Q) (\alpha \cdot Q) \Big],\\ 
& \dom\big(\hatt A\big)= [C_0^{\infty} (\bbR^n \backslash \{0\})]^N.  \lb{3.12a}
\end{split}
\end{align}
In this context we note that $C_0^{\infty}(\bbR^n \backslash \{0\})$, $n \in \bbN$, $n \geq 2$, 
is a core for $\nabla$ (see, e.g., \cite[p.~97]{Fa75}).
Therefore, the operator $\hatt A$ is a first-order differential operator of the form introduced in \eqref{2.9} with $F_j$, $1\leq j\leq n$,  defined by 
\begin{equation}\label{def_Fj_special}
F_j(p)= \frac12 (\alpha \cdot p) |p|^{-2} \eta(p) p_j, \quad p\in \bbR^n\backslash\{0\}, 
\; 1 \leq j \leq n.
\end{equation}
Given Hypothesis \ref{h3.2} on $\eta$, one can check that $F_j$ leaves the core $C_0^{\infty}(\bbR^n \backslash \{0\})$ invariant (in particular, the $|x|$-behavior of $\eta(Q)$ near $x = 0$ 
is not felt by functions in $C_0^{\infty}(\bbR^n \backslash \{0\})$). Moreover, Hypothesis \ref{h3.2} 
also guarantees that the partial derivatives of $F_j$ are bounded functions, and therefore 
\begin{equation}\label{F_j_partial}
\overline{\big[-i \partial_k, F_j\big]}\in \cB\Big(\LN\Big),\quad 1\leq j,k \leq n.
\end{equation}
Thus, the assumptions of Hypothesis \ref{h2.1} are satisfied, and therefore, Theorem \ref{t2.4} implies that the operator $\hatt A$ is essentially self-adjoint. The inclusion \eqref{2.11zzz} 
then follows from \eqref{a.10s}. 
\end{proof}

In the next result, we compute the commutator of $H_0$ with $A$.

\begin{proposition}\lb{p3.4}
Assume Hypotheses \ref{h3.1} and \ref{h3.2}, then
\begin{equation}\lb{A.31z}
i[A,H_0]\psi = B\psi, \quad \psi\in \cD_0 (\bbR^n). 
\end{equation}
Therefore, for each $\psi_1,\psi_2\in \LoneN\cap \WoneN$, 
\begin{equation}\lb{A.31zz}
(\psi_1,B\psi_2)_{\LNs} = \big(iH_0\psi_1,\ol A\psi_2\big)_{\LNs} - \big(i\ol A\psi_1,H_0\psi_2\big)_{\LNs}.
\end{equation}
\end{proposition}
\begin{proof}
We compute the commutator of $\widehat{H}_0 = \alpha\cdot Q$ and 
\begin{align}\lb{3.15zz}
\widehat{A} = \frac{1}{2} \Big( \widehat{H}_0 |Q|^{-2}h(|Q|) (Q\cdot (-i\nabla)) + ((-i\nabla)\cdot Q)\widehat{H}_0|Q|^{-2}h(|Q|)\Big)
\end{align}
on $[C_0^{\infty}(\bbR^n\backslash\{0\})]^N$.  To this end, one observes that
\begin{align}
\big[\widehat{H}_0,Q\cdot (-i\nabla)\big] &= \big[\widehat{H}_0,(-i\nabla)\cdot Q\big]\no\\
&= \sum_{j,k=1}^n [\alpha_jQ_j,Q_k(-i\partial_k)]\no\\
&= -i\sum_{j=1}^n Q_j\alpha_j[Q_j,\partial_j]\no\\
&= i\widehat{H}_0
\end{align}
on $[C_0^{\infty}(\bbR^n\backslash\{0\})]^N$.  As a result, one obtains
\begin{align}
\big[\widehat{H}_0,\widehat{A}\big] &= 2^{-1}\widehat{H}_0|Q|^{-2}h(|Q|)\big[\widehat{H}_0,Q\cdot(-i\nabla)\big] + 2^{-1}\big[\widehat{H}_0,(-i\nabla)\cdot Q\big]\widehat{H}_0|Q|^{-2}h(|Q|)\no\\
&= (i/2) \widehat{H}_0|Q|^{-2}h(|Q|)\widehat{H}_0 + (i/2) \widehat{H}_0\widehat{H}_0|Q|^{-2}h(|Q|)\no\\
&=ih(|Q|)\no\\
&=i\widehat{B}
\end{align}
on $[C_0^{\infty}(\bbR^n\backslash\{0\})]^N$, where
\begin{equation}
\widehat{B} := h(|Q|).
\end{equation}
\end{proof}

Assuming Hypothesis \ref{h3.2}, the square root $B^{1/2}$ is defined by the spectral theorem via
\begin{equation}
B^{1/2} = \eta(-i \nabla)^{1/2} \in \cB\Big(\LN\Big).
\end{equation}

The next result requires Kato's inequality in $\bbR^n$, $n \geq 2$, which is of the form (cf., e.g., \cite[p.~19]{BE11}, \cite{He77})
\begin{align}
\begin{split} 
K_n \int_{\bbR^n} d^n x \, |x|^{-1}|f(x)|^2 \leq \int_{\bbR^n} d^n p \, |p| \big|\widehat f(p)\big|^2,\quad f \in \dom((-\Delta)^{1/4}), \; n \in \bbN, \; n \geq 2,& \lb{2.20a}
\end{split} 
\end{align}
for some constants $K_n > 0$. In particular,
\begin{align}
K_n \int_{\bbR^n} d^n x \, |x|^{-1} |f(x)|^2 \leq 
\int_{\bbR^n} d^n p \, \langle p \rangle \big|\widehat f(p)\big|^2, 
\quad \widehat f \in L^2_{1/2}(\bbR^n), \; n \in \bbN, \; n \geq 2,& \lb{2.21a}
\end{align}
equivalently, and in the form to be used below,
\begin{align}
K_n \int_{\bbR^n} d^n p \, |p|^{-1} \big|\widehat f(p)\big|^2 \leq 
\int_{\bbR^n} d^n x\, \langle x \rangle |f(x)|^2, 
\quad f \in L^2_{1/2}(\bbR^n), \; n \in \bbN, \; n \geq 2.&    \lb{2.22a}
\end{align}

\begin{proposition}\lb{p3.6}
Assume Hypothesis \ref{h3.2}. ~Then
\begin{equation}
B^{-1/2}\in \cB\Big(\LonehN,\LN\Big),
\end{equation}
that is, there exists 
$C_1>0$ such that
\begin{equation}\lb{3.26zz}
\|B^{-1/2} \psi\|_{\LNs} \leq C_1\|\psi\|_{\LonehNs},\quad \psi \in \LonehN.
\end{equation}
\end{proposition}
\begin{proof}
The nonnegative function $\eta(p)^{-1/2}$ of $ p    =(p_1,\ldots,p_n)$ satisfies
\begin{equation}\lb{3.22x}
\eta(p)^{-1/2} \leq 2^{1/2} + \frac{1}{|p|^{1/2}},\quad  p \neq 0,
\end{equation}
so it suffices to note that
\begin{equation}
|-i \nabla|^{-1/2}\in \cB\Big(\LonehN,\LN\Big),
\end{equation}
which follows from Kato's inequality in the form of \eqref{2.22a}, 
\begin{align}
\begin{split} 
\big\||-i \nabla|^{-1/2} \psi \big\|_{\LNs} = 
\big\||p|^{-1/2} \widehat \psi \big\|_{\LNs} \leq K_n^{-1/2} \|\psi\|_{\LonehNs},&  \\
\quad \psi\in \LonehN, \; n \geq 2.&    \lb{2.29} 
\end{split} 
\end{align} 
\end{proof}

\begin{proposition}\lb{p3.7}
Assume Hypotheses \ref{h3.1} and \ref{h3.2}. ~Then there exists a constant $C_2>0$ such that
\begin{equation}
\big\|B^{-1/2}A\psi\big\|_{\LNs}\leq C_2\|\psi\|_{\LoneNs},\quad \psi \in \dom(A)=\cD_0(\bbR^n).
\end{equation}
Thus, the operator $B^{-1/2}A:\cD_0 (\bbR^n) \to \LN$ extends to an element
\begin{equation} 
B^{-1/2}\overline{A}\in \cB\Big(\LoneN,\LN\Big).  
\end{equation} 
\end{proposition}
\begin{proof}
With $\widehat{A}$ defined by \eqref{3.15zz}, write
\begin{equation}\lb{3.35v}
\widehat{A} = \frac{1}{2}\big( \widehat{A}_{\ell} + \widehat{A}_r\big),
\end{equation}
where
\begin{align}
\widehat{A}_{\ell} &= \widehat{H}_0|Q|^{-2}h(|Q|) (Q\cdot (-i\nabla)),\\
\widehat{A}_r &= ((-i\nabla)\cdot Q) \widehat{H}_0|Q|^{-2}h(|Q|),
\end{align}
on $[C_0^{\infty}(\bbR^n\backslash\{0\})]^N$.  Then one observes that
\begin{align}
\widehat{A}_r &= -i\sum_{j=1}^n \partial_jQ_j \widehat{H}_0|Q|^{-2}h(|Q|)\no\\
&= -i\sum_{j=1}^n\big[\partial_j,Q_j\widehat{H}_0|Q|^{-2}h(|Q|)\big] + \widehat{A}_{\ell}.
\end{align}
One notes that the operator $Q_j\hatt H_0 |Q|^{-2} h(|Q|)$ represents the operator $F_j(Q)$, where 
$F_j$ is defined according to \eqref{def_Fj_special}. Hence, by \eqref{F_j_partial} it follows that the commutator 
$[\partial_j, Q_j \hatt H_0 |Q|^{-2} h(|Q|)]$  extends to a bounded operator for every $j=1,\dots, n$. Therefore, the operator $\hatt A_0$ defined by
\begin{equation}\lb{3.39v}
\widehat{A}_0:= -i\sum_{j=1}^n\big[\partial_j,Q_j\widehat{H}_0|Q|^{-2}h(|Q|)\big],\quad \dom(\widehat{A}_0)=[C_0^{\infty}(\bbR^n\backslash\{0\})]^N, 
\end{equation}
extends to a bounded operator
\begin{equation}
\overline{\widehat{A}_0}: [L^2(\bbR^n)]^N\to [L^2(\bbR^n)]^N
\end{equation}
commuting with $\widehat{B} = h(|Q|)$.  In particular, one infers that
\begin{equation}\lb{3.39xx}
\widehat{A} = \widehat{A}_{\ell} + \frac{1}{2}\overline{\widehat{A}_0}: [C_0^{\infty}(\bbR^n\backslash\{0\})]^N\to [L^2(\bbR^n)]^N.
\end{equation}
Writing
\begin{equation}
A_{\ell} = (\alpha \cdot (-i \nabla)    )(- \Delta)^{-1} \eta(-i \nabla)((-i \nabla) \cdot  Q   ),
\end{equation}
one obtains
\begin{align}
\big\|B^{-1/2}A_{\ell}\psi\big\|_{\LNs} &= \big\|(\alpha \cdot (-i \nabla))(- \Delta)^{-1} \eta(-i \nabla)^{1/2}((-i \nabla) \cdot  Q   )\psi\big\|_{\LNs}    \no\\
& \leq C_{\ell}\|\psi\|_{\LoneNs}, \quad \psi\in \cD_0 (\bbR^n),\lb{3.40zz}
\end{align}
for an appropriate constant $C_{\ell}>0$, by the (bounded) functional calculus for self-adjoint operators.  Finally, one uses \eqref{3.40zz}, the fact that $\overline{\widehat{A}_0}$ commutes with $\widehat{B}$, boundedness of $\overline{\widehat{A}_0}$, and Proposition \ref{p3.6}, to obtain
\begin{align}
\big\|B^{-1/2}A\psi\big\|_{\LNs} &\leq \big\|B^{-1/2}A_{\ell}\psi\big\|_{\LNs} 
+ 2^{-1}\Big\| \overline{\widehat{A}_0}\widehat{B}^{-1/2}\widehat{\psi}\Big\|_{\LNs} 
\no\\
&\leq C_2 \|\psi\|_{\LoneNs}, \quad \psi\in \dom(A),
\end{align}
for an appropriate constant $C_2>0$.
\end{proof}

Next, we investigate the commutator of $A$ with $B$. 

\begin{proposition}\lb{p3.8}
Assume Hypotheses \ref{h3.1} and \ref{h3.2}.  Then $B$ leaves the domain of $\overline{A}$ invariant, and there exists $K\in \cB\Big(\LN\Big)$, which commutes with $B$, such that
\begin{equation}\lb{A.19}
i[B,\overline{A}]\psi = KB\psi,\quad \psi\in \dom(\overline{A}).
\end{equation}
Thus, $i[B,\overline{A}]$ extends to a bounded operator in $\cB\Big(\LN\Big)$.

Moreover, if 
$\psi_1,\psi_2\in \LoneN$, then
\begin{equation}\lb{A.20}
(\psi_1, KB\psi_2)_{\LNs} = (i\ol A\psi_1,B\psi_2)_{\LNs} 
- (iB\psi_1,\ol A\psi_2)_{\LNs}.
\end{equation}
\end{proposition}
\begin{proof}
Using \eqref{3.39xx}, one computes
\begin{align}
i\big[\widehat{B},\widehat{A}\big] &= i \big[\widehat{B},\widehat{A}_{\ell}+(1/2)\widehat{A}_0\big]\no\\
&= i\big[\widehat{B},\widehat{A}_{\ell}\big]\no\\
&= -\widehat{H}_0|Q|^{-2}h(|Q|)\sum_{j=1}^n Q_j[\partial_j,h(|Q|)]\no\\
&= -\widehat{H}_0h(|Q|)h'(|Q|)|Q|^{-1}
\end{align}
on $[C_0^{\infty}(\bbR^n\backslash\{0\})]^N$.  The claim now follows as 
\begin{equation}
-\widehat{H}_0h'(|Q|)|Q|^{-1}: [C_0^{\infty}(\bbR^n\backslash\{0\})]^N\to \LN
\end{equation}
extends to a bounded operator on $\LN$ by Hypothesis \ref{h3.2} and since the set $[C_0^{\infty}(\bbR^n\backslash\{0\})]^N$ is a core for $\widehat{A}$ by definition.
\end{proof}

Assuming Hypotheses \ref{h3.1} and \ref{h3.2}, define
\begin{equation}
T_{\varepsilon}^{\pm}(\lambda,\mu) = H_0 - (\lambda\pm i\mu)I_{\LNs} \mp i\varepsilon B,\quad \lambda\in \bbR,\, \mu>0,\, \varepsilon \in [0,1).
\end{equation}

\begin{proposition}\lb{p3.9}
Assume Hypotheses \ref{h3.1} and \ref{h3.2}.  For each $\lambda\in \bbR$, $\mu\in (0,\infty)$, 
$\varepsilon\in [0,1)$, $T_{\varepsilon}^{\pm}(\lambda,\mu)$ is a linear homeomorphism 
from $\WoneN$ to $\LN$ satisfying 
\begin{equation}\lb{A.84aa}
\|T_{\varepsilon}^{\pm}(\lambda,\mu)\psi\|_{\LNs}^2 \geq 
\mu^2\|\psi\|_{\LNs}^2,\quad \psi\in \WoneN.
\end{equation}
\end{proposition}
\begin{proof}
Let $\lambda\in \bbR$, $\mu\in (0,\infty)$, and $\varepsilon \in [0,1)$.  The claim is evident if $\varepsilon = 0$, so we assume from now on that $\varepsilon \neq 0$.
To obtain the lower bound in \eqref{A.84aa}, let $\psi\in \WoneN$.  Since $H_0$ is self-adjoint and $B$ is nonnegative,
\begin{align}
&\Re\big\{\big([H_0 - \lambda I_{\LNs} \mp i\varepsilon B]\psi,\mp i\mu \psi\big)_{\LNs} \big\}\no\\
&\quad = \Re\big\{ \mp i\mu \big([H_0 - \lambda I_{\LNs}]\psi,\psi\big)_{\LNs} \big\} + \Re\big\{ \varepsilon \mu\big(B\psi,\psi\big)_{\LNs} \big\}\no\\
&\quad = \varepsilon \mu (\psi,B\psi)_{\LNs}\geq 0,\quad \psi\in \WoneN,
\end{align}
and therefore,
\begin{align}
&\|T_{\varepsilon}^{\pm}(\lambda,\mu)\psi\|_{\LNs}^2\no\\
& \quad = \| [H_0 - \lambda I_{\LNs} \mp i\varepsilon B]\psi \|_{\LNs}^2 
+ \mu^2 \|\psi\|_{\LNs}^2 \no\\
&\qquad + 2\, \Re\big\{\big([H_0 - \lambda I_{\LNs} \mp i\varepsilon B] 
\psi,\mp i\mu \psi\big)_{\LNs} \big\}\no\\
& \quad \geq \mu^2\|\psi\|_{\LNs}^2,\quad \psi\in \WoneN.
\end{align}

It is clear that
\begin{equation}
T_{\varepsilon}^{\pm}(\lambda,\mu)\in \cB\Big(\WoneN,\LN\Big)
\end{equation}
for each $\lambda\in \bbR$ and $\mu\in (0,\infty)$, since the operator $i\varepsilon B$ is bounded.  Moreover, $T_{\varepsilon}^{\pm}(\lambda,\mu)$ is a closed bijection.  In fact, \eqref{A.84aa} immediately implies that $T_{\varepsilon}^{\pm}(\lambda,\mu)$ is an injection with a closed range.  In addition, since $[T_{\varepsilon}^{\pm}(\lambda,\mu)]^*=T_{\varepsilon}^{\mp}(\lambda,\mu)$, $T_{\varepsilon}^{\pm}(\lambda,\mu)$ is a surjection.
\end{proof}

By Proposition \ref{p3.9}, $T_{\varepsilon}^{\pm}(\lambda,\mu)$ is a linear homeomorphism from 
$\WoneN$ to $\LN$ for each $\lambda\in \bbR$, $\mu\in (0,\infty)$, $\varepsilon\in [0,1)$.  Therefore, we define
\begin{align}
& G_{\varepsilon}^{\pm}(\lambda,\mu) := (T_{\varepsilon}^{\pm}(\lambda,\mu))^{-1}\in \cB\Big(\LN,\WoneN\Big)\subset \cB\Big(\LN\Big),   \no \\
& \hspace*{6.5cm} \lambda\in \bbR,\, \mu\in (0,\infty),\, \varepsilon\in[0,1),
\end{align}
noting that $G_{\varepsilon}^{\pm}(\lambda,\mu)$ are adjoint to one another,
\begin{equation}
(G_{\varepsilon}^{\pm}(\lambda,\mu))^* = G_{\varepsilon}^{\mp}(\lambda,\mu),\quad \lambda\in \bbR,\, \mu\in (0,\infty),\,\varepsilon\in[0,1).
\end{equation}
We also define the scalar-valued function
\begin{align}
&F_{\varepsilon}^{\pm}(\lambda,\mu;\psi) := (\psi,G_{\varepsilon}^{\pm}(\lambda,\mu)\psi)_{\LNs},\\
& \quad \lambda\in \bbR,\, \mu\in (0,\infty),\, \varepsilon\in[0,1),\, \psi\in\LN. 
\no
\end{align}
The following result is a standard application of the second resolvent identity, so we omit the details of its proof.

\begin{proposition}\lb{p3.10}
Assume Hypotheses \ref{h3.1} and \ref{h3.2}.  For each $\psi\in \LN$, $\lambda\in \bbR$, $\mu\in (0,\infty)$, the function
\begin{align}
F_{\cdot}^{\pm}(\lambda,\mu;\psi):
\begin{cases}
[0,1)\to \bbC,\\
\varepsilon\mapsto F_{\varepsilon}^{\pm}(\lambda,\mu;\psi),
\end{cases}
\end{align}
is smooth on $[0,1)$ and
\begin{equation}\lb{3.60zz}
\frac{dF_{\varepsilon}^{\pm}(\lambda,\mu;\psi)}{d\varepsilon} = \pm(\psi,iG_{\varepsilon}^{\pm}(\lambda,\mu)BG_{\varepsilon}^{\pm}(\lambda,\mu)\psi)_{\LNs},\quad \varepsilon\in [0,1).
\end{equation}
\end{proposition}

\begin{proposition}\lb{p3.13}
Assume Hypotheses \ref{h3.1} and \ref{h3.2}.  $G_{\varepsilon}^{\pm}(\lambda,\mu)$ leaves $\LoneN$ invariant:
\begin{equation}
G_{\varepsilon}^{\pm}(\lambda,\mu)\LoneN\subset \LoneN.
\end{equation}
In addition, the following estimates hold:
\begin{align}
\big\|B^{1/2}G_{\varepsilon}^{\pm}(\lambda,\mu)\psi\big\|_{\LNs}^2 &\leq \varepsilon^{-1}|F_{\varepsilon}^{\pm}(\lambda,\mu;\psi)|,\quad \psi\in \LN,     \lb{A.86}\\
\big\|B^{1/2}G_{\varepsilon}^{\pm}(\lambda,\mu)\phi\big\|_{\LNs} &\leq C_1\varepsilon^{-1} 
\|\phi\|_{\LonehNs},\quad \phi \in \LonehN,   \lb{A.34}\\
&\hspace*{2.3cm} \lambda\in \bbR,\, \mu\in(0,\infty),\, \varepsilon\in(0,1),\no
\end{align}
where $C_1$ is the same constant as in \eqref{3.26zz}.
\end{proposition}
\begin{proof}
To prove the invariance claim, it suffices to show that $G_{\varepsilon}^{\pm}(\lambda,\mu)$ preserves the domain of the unbounded self-adjoint operator $Q_j$ for all $1\leq j\leq n$.  Thus, let $j\in \bbN$, with $1\leq j\leq n$, be fixed. Note that the sequence of bounded operators
\begin{equation}
\Big\{ i\big(iI_{\LNs} +m^{-1}Q_j\big)^{-1}\Big\}_{m=1}^{\infty}
\end{equation}
converges strongly to the identity operator in $\LN$.  One computes
\begin{align}
&\Big[iQ_j\big(iI_{\LNs} +m^{-1}Q_j\big)^{-1},G_{\varepsilon}^{\pm}(\lambda,\mu)\Big]\no\\
&\quad = G_{\varepsilon}^{\pm}(\lambda,\mu)\Big[H_0 \mp i\varepsilon B, iQ_j\big(iI_{\LNs} + m^{-1}Q_j\big)^{-1}\Big]G_{\varepsilon}^{\pm}(\lambda,\mu)\no\\
&\quad = G_{\varepsilon}^{\pm}(\lambda,\mu)\Big[H_0 \mp i\varepsilon B,m\big(iI_{\LNs}+m^{-1}Q_j\big)^{-1}\Big]G_{\varepsilon}^{\pm}(\lambda,\mu).\lb{3.66xx}
\end{align}
In addition, the operator $B$ preserves the domain of $Q_j$ and $[Q_j,B]:\dom(Q_j)\to \LN$ extends to a bounded operator $q_j(B)\in \cB\Big(\LN\Big)$.  Therefore,
\begin{align}
&m\Big[H_0\mp i\varepsilon B,\big(iI_{\LNs}+m^{-1}Q_j\big)^{-1}\Big]\no\\
&\quad = i\alpha_j\big(iI_{\LNs}+m^{-1}Q_j\big)^{-1}\no\\
&\qquad \mp i\varepsilon \big(iI_{\LNs}+m^{-1}Q_j\big)^{-1}q_j(B)\big(iI_{\LNs}+m^{-1}Q_j\big)^{-1},\lb{3.66zz}
\end{align}
which converges strongly as $m\to \infty$ to
\begin{equation}
-i\alpha_j \pm i\varepsilon q_j(B).
\end{equation}
The fact that $G_{\varepsilon}^{\pm}(\lambda,\mu)$ preserves the domain of $Q_j$ now follows since the righthand side of \eqref{3.66xx} converges strongly to
\begin{equation}
-G_{\varepsilon}^{\pm}(\lambda,\mu)[i\alpha_j \mp i\varepsilon q_j(B)] G_{\varepsilon}^{\pm}(\lambda,\mu).
\end{equation}

To prove \eqref{A.86}, let $\psi\in \LN$, $\lambda\in \bbR$, and $\mu\in (0,\infty)$ be fixed.  Then
\begin{align}
&\big|F_{\varepsilon}^{\pm}(\lambda,\mu;\psi)\big|  
\geq \pm \Im\big[\big(\psi,G_{\varepsilon}^{\pm}(\lambda,\mu)\psi\big)_{\LNs}\big]\no\\
&\quad = \pm\Im\big[\big((H_0-(\lambda\pm i\mu)I_{\LNs}\mp i\varepsilon B)G_{\varepsilon}^{\pm}(\lambda,\mu)\psi,G_{\varepsilon}^{\pm}(\lambda,\mu)\psi\big)_{\LNs}\big]\no\\
&\quad = \mu\big\|G_{\varepsilon}^{\pm}(\lambda,\mu)\psi\big\|_{\LNs}^2 + \varepsilon \big\|B^{1/2}G_{\varepsilon}^{\pm}(\lambda,\mu)\psi\big\|_{\LNs}^2
\no\\
&\quad \geq \varepsilon \big\|B^{1/2}G_{\varepsilon}^{\pm}(\lambda,\mu)\psi\big\|_{\LNs}^2,\lb{2.59}
\end{align}
which yields \eqref{A.86}.  Next,
\begin{align}
&\big|F_{\varepsilon}^{\pm}(\lambda,\mu,\phi)\big|= 
\big|\big(\phi,G_{\varepsilon}^{\pm}(\lambda,\mu)\phi\big)_{\LNs}\big|\no\\
& \quad \leq \big\|B^{-1/2}\phi\big\|_{\LNs}\big\|B^{1/2}G_{\varepsilon}^{\pm}(\lambda,\mu)\phi\big\|_{\LNs}\no\\
& \quad \leq C_1\|\phi\|_{\LonehNs}\big\|B^{1/2}G_{\varepsilon}^{\pm}(\lambda,\mu)\phi\big\|_{\LNs},\quad \phi\in \LonehN.\lb{2.60}
\end{align}
where the last estimate makes use of Proposition \ref{p3.6}.
Upon combining \eqref{2.59} and \eqref{2.60}, one obtains
\begin{align}
&\varepsilon\big\|B^{1/2}G_{\varepsilon}^{\pm}(\lambda,\mu)\phi\big\|_{\LNs}^2  \leq \big|F_{\varepsilon}^{\pm}(\lambda,\mu;\phi)\big|\no\\
&\quad \leq C_1\|\phi\|_{\LonehNs}\big\|B^{1/2}G_{\varepsilon}^{\pm}(\lambda,\mu)\phi\big\|_{\LNs},
\end{align}
which then implies
\begin{equation}
\big\|B^{1/2}G_{\varepsilon}^{\pm}(\lambda,\mu)\phi\big\|_{\LNs} \leq 
C_1\varepsilon^{-1}\|\phi\|_{\LonehNs},\quad \phi\in \LonehN.
\end{equation}
and \eqref{A.34} follows.
\end{proof}

The main goal of this section is to prove a uniform limiting absorption principle for $H_0$, namely, 
given $\lambda\in \bbR$, $\mu\in(0,\infty)$, there exists a $C \in (0, \infty)$ (independent of 
$\lambda\in \bbR$, $\mu\in(0,\infty)$), such that
\begin{align}
\big\| (H_0 - (\lambda\pm i\mu)I_{\LNs})^{-1} \big\|_{\cB(\LoneNs,\LmoneNs)} 
\leq C,&\\
\quad \lambda\in \bbR,\; \mu\in(0,\infty).\no&
\end{align}
To accomplish this, it actually suffices to prove 
\begin{equation}
\big|F_0^{\pm}(\lambda,\mu;\psi)\big| \leq C'\|\psi\|_{\LoneNs}^2,\quad \psi\in \LoneN, \; \lambda\in \bbR,\; \mu\in (0,\infty),
\end{equation}
for some constant $C' \in (0, \infty)$ (independent of $\lambda\in \bbR$, $\mu\in(0,\infty)$), 
by the next lemma. The latter is surely well-known, but we include a proof for completeness.  

\begin{lemma}\lb{l3.14}
Suppose that $\cH_+$ is a Banach space that embeds continuously and densely into the Hilbert space $\cH$ so that $\cH$ embeds continuously and densely into $\cH_-:=\cH_+^*$.  If $H:\dom(H)\subset \cH \to \cH$ is a self-adjoint operator, then 
\begin{equation}\lb{A.117a}
\big\|(H-(\lambda\pm i\mu)I_{\cH})^{-1}\big\|_{\cB(\cH_+,\cH_-)}\leq C_1, 
\quad \lambda\in \bbR,\; \mu\in(0,\infty),
\end{equation}
for some constant $C_1 \in (0, \infty)$ $($independent of $\lambda\in \bbR$, $\mu\in(0,\infty)$$)$ 
if and only if
\begin{equation}\lb{A.118a}
\big|(\psi,(H-(\lambda \pm i\mu)I_{\cH})^{-1}\psi)_{\cH} \big|\leq C_2\|\psi\|_{\cH_+}^2,\quad \psi\in \cH_+,\, \lambda\in \bbR,\; \mu\in(0,\infty),
\end{equation}
for some constant $C_2 \in (0,\infty)$ $($independent of $\lambda\in \bbR$, $\mu\in(0,\infty)$$)$.
\end{lemma}
\begin{proof}
The equivalence of \eqref{A.117a} and \eqref{A.118a} follows from the estimates
\begin{equation}\lb{3.76zz}
\sup_{\psi\in \cH_+,\|\psi\|_{\cH_+}=1}|(\psi,T\psi)_{\cH}| \leq \|T\|_{\cB(\cH_+,\cH_-)}\leq 4 \sup_{\psi\in \cH_+,\|\psi\|_{\cH_+}=1}|(\psi,T\psi)_{\cH}|,
\end{equation}
which hold for any $T\in \cB(\cH_+,\cH)$.  In fact, for $T\in \cB(\cH_+,\cH)$, one obtains
\begin{align}
\|T\|_{\cB(\cH_+,\cH_-)} &= \sup_{\psi\in\cH_+,\|\psi\|_{\cH_+}=1}\|T\psi\|_{\cH_-}\no\\
&= \sup_{\psi\in \cH_+,\|\psi\|_{\cH_+}=1}\bigg(\sup_{\phi\in \cH_+,\|\phi\|_{\cH_+}=1}|(T\psi)(\phi)| \bigg)\no\\
&= \sup_{\psi,\phi\in \cH_+,\|\psi\|_{\cH_+}=\|\phi\|_{\cH_+}=1}|(\phi,T\psi)_{\cH}|,
\end{align}
and \eqref{3.76zz} follows by the polarization principle.
\end{proof}

The following lemma will be used in the proof of the global limiting absorption principle for $H_0$.

\begin{lemma}[Lemma 7.A.1 in \cite{ABG96}]\lb{lA.1}
Let $(a,b)\subset \bbR$ be an open interval and let $f$, $\phi$, and $\psi$ be nonnegative real functions on $(a,b)$ with $f$ bounded and measurable and $\phi,\psi\in L^1((a,b);dx)$.  Assume there exist constants $\omega\in [0,\infty)$ and $\theta\in [0,1)$ with
\begin{equation}
f(\lambda)\leq \omega +\int_{\lambda}^b dt \, [\phi(t)f(t)^{\theta} + \psi(t)f(t)], 
\quad \lambda\in (a,b).
\end{equation}
Then
\begin{equation}
f(\lambda)\leq e^{\int_{\lambda}^b dt \, \psi(t)}\cdot\Bigg[\omega^{1-\theta} 
+ (1-\theta)\int_{\lambda}^b ds \, \phi(s)e^{(\theta-1)\int_s^b dt \, \psi(t)}\Bigg]^{1/(1-\theta)},\quad \lambda\in(a,b).
\end{equation}
\end{lemma}

Given these preparations, everything is finally in place to state and prove the principal result of this section, a global limiting absorption principle for $H_0$ in all dimensions $n \in \bbN$, 
$n \geq 2$: 

\begin{theorem} \lb{t3.15} 
Assume Hypotheses \ref{h3.1} and \ref{h3.2} and let $\lambda\in \bbR$, $\mu\in(0,\infty)$. Then 
\begin{equation}\lb{2.13}
\big\| (H_0 - (\lambda\pm i\mu)I_{\LNs})^{-1} \big\|_{\cB(\LoneNs,\LmoneNs)} 
\leq C 
\end{equation}
for some constant $C \in (0, \infty)$, independent of $\lambda\in \bbR$, $\mu\in(0,\infty)$. 
Equivalently,
\begin{equation}
\big\|\langle Q \rangle^{-1} (H_0 - (\lambda\pm i\mu)I_{\LNs})^{-1} 
\langle Q \rangle^{-1}\big\|_{\cB(\LNs)} \leq C, \quad \lambda\in \bbR, \; \mu\in(0,\infty). \lb{2.76}
\end{equation}
Consequently, $\langle Q \rangle^{-1}$ is $H_0$-Kato-smooth, that is, given $\varepsilon_0 > 0$, 
for each $f \in \LN$, 
\begin{align}
& \sup_{\varepsilon \in(0, \varepsilon_ 0), \, \|f\|_{\LNs}=1} \f{1}{4 \pi^2} \int_{\bbR} d \lambda \, 
\Big[\big\|\langle Q \rangle^{-1}(H_0 - (\lambda + i \varepsilon) I_{\LNs})^{-1}f\big\|^2_{\LNs}    \no \\
& \hspace*{2.5cm} + \big\|\langle Q \rangle^{-1}(H_0 - (\lambda - i \varepsilon) I_{\LNs})^{-1}f\big\|^2_{\LNs}\Big] < \infty.     \lb{3.1} 
\end{align}
\end{theorem}
\begin{proof}
By Lemma \ref{l3.14}, it suffices to show 
\begin{equation}\lb{2.14}
\big|F_0^{\pm}(\lambda,\mu;\psi)\big| \leq C\|\psi\|_{\LoneNs}^2,\quad \psi\in \LoneN,\; 
\lambda\in \bbR,\; \mu\in (0,\infty),
\end{equation}
for some constant $C>0$.

Choose $\psi_1=G_{\varepsilon}^{\pm}(\lambda,\mu)^*\psi$ and $\psi_2=G_{\varepsilon}^{\pm}(\lambda,\mu)\psi$. On one hand, since $\psi\in  \LoneNs$, Proposition \ref{p3.13} guarantees that $\psi_1, \psi_2\in \LoneNs$. On the other hand, $\psi_1, \psi_2\in \WoneN$ since
\begin{equation}
G_{\varepsilon}^{\pm}(\lambda,\mu)\in \cB\Big(\LN,\WoneN\Big).
\end{equation}
Therefore, combining the results of  Propositions \ref{p3.4} and \ref{p3.8}, one obtains 
\begin{align}
\frac{d}{d\varepsilon} F_{\varepsilon}^{\pm}(\lambda,\mu;\psi)&\stackrel{\eqref{3.60zz}}{=} \pm i(G_{\varepsilon}^{\mp}(\lambda,\mu)\psi,BG_{\varepsilon}^{\pm}(\lambda,\mu)\psi)_{\LNs}\no\\
&\stackrel{\eqref{A.31zz}}{=}\hspace{.2mm}\pm (H_0G_{\varepsilon}^{\mp}(\lambda,\mu)\psi,\overline{A}G_{\varepsilon}^{\pm}(\lambda,\mu)\psi)_{\LNs}\no\\
&\hspace{8.5mm}  \mp (\overline{A}G_{\varepsilon}^{\mp}(\lambda,\mu)\psi,H_0G_{\varepsilon}^{\pm}(\lambda,\mu)\psi)_{\LNs}\no\\
&\hspace{2.7mm} =\hspace{2mm}\pm(\psi,\overline{A}G_{\varepsilon}^{\pm}(\lambda,\mu)\psi)_{\LNs} \mp (\overline{A}G_{\varepsilon}^{\mp}(\lambda,\mu)\psi,\psi)_{\LNs}\no\\
&\hspace{8.5mm} -(i\varepsilon BG_{\varepsilon}^{\mp}(\lambda,\mu)\psi,\overline{A}G_{\varepsilon}^{\pm}(\lambda,\mu)\psi)_{\LNs}\no\\
&\hspace{8.5mm} - (\overline{A}G_{\varepsilon}^{\mp}(\lambda,\mu)\psi,i\varepsilon BG_{\varepsilon}^{\pm}(\lambda,\mu)\psi)_{\LNs}\no\\
& \stackrel{\eqref{A.20}}{=} \pm(\overline{A}\psi,G_{\varepsilon}^{\pm}(\lambda,\mu)\psi)_{\LNs} \mp (G_{\varepsilon}^{\mp}(\lambda,\mu)\psi,\overline{A}\psi)_{\LNs}\no\\
&\hspace{8mm}+ \varepsilon\big(KB^{1/2}G_{\varepsilon}^{\mp}(\lambda,\mu)\psi,B^{1/2}G_{\varepsilon}^{\pm}(\lambda,\mu)\psi\big)_{\LNs},\\
&\hspace*{1.2cm} \psi\in \LoneN,\; \lambda\in \bbR,\; \mu\in (0,\infty),\; \varepsilon \in (0,1).\no
\end{align}

Hence,
\begin{align}
&\bigg|\frac{d}{d\varepsilon} F_{\varepsilon}(\lambda,\mu;\psi)\bigg|   
\leq \big\|B^{-1/2}\overline{A}\psi\big\|_{\LNs} \Big[\big\|B^{1/2}G_{\varepsilon}^{\pm}(\lambda,\mu)^*\psi\big\|_{\LNs}\no\\
&\qquad + \big\|B^{1/2}G_{\varepsilon}^{\pm}(\lambda,\mu)\psi\big\|_{\LNs}\Big]\no\\
&\qquad + \|K\|_{\cB(\LNs)} \, \varepsilon \, \big\|B^{1/2}G_{\varepsilon}^{\pm}(\lambda,\mu)^*\psi\big\|_{\LNs}\big\|B^{1/2}G_{\varepsilon}^{\pm}(\lambda,\mu)\psi\big\|_{\LNs}\no\\
&\quad \leq C_4\varepsilon^{-1/2} \big|F_{\varepsilon}^{\pm}(\lambda,\mu;\psi)\big|^{1/2}\|\psi\|_{\LoneNs}, \lb{3.84y} \\
& \qquad \;\, \psi\in \LoneN,\, \lambda\in \bbR,\, \mu\in (0,\infty),\, \varepsilon \in (0,1), \no
\end{align}
where we applied Propositions \ref{p3.7} and \ref{p3.13}, and 
$C_4 \in (0,\infty)$ is a constant independent of $\lambda\in \bbR$, 
$\mu\in (0,\infty)$, and $\varepsilon\in (0,1)$.

Next, fix $r_0\in (0,1)$, and integrate over $[r,r_0]\subset(0,1)$ in \eqref{3.84y} to obtain for arbitrary $\psi\in \LoneN$,
\begin{align}
&C_4\|\psi\|_{\LoneNs}\int_r^{r_0} \, d\varepsilon \, \varepsilon^{-1/2}|F_{\varepsilon}^{\pm}(\lambda,\mu;\psi)|^{1/2} \\
&\quad \geq \int_r^{r_0} \, d\varepsilon \,\bigg|\frac{d}{d\varepsilon} F_{\varepsilon}^{\pm}(\lambda,\mu;\psi) \bigg| \no\\
&\quad\geq \Bigg| \int_r^{r_0} \, d\varepsilon \, \frac{d}{d\varepsilon} F_{\varepsilon}^{\pm}(\lambda,\mu;\psi) \Bigg|\no\\
&\quad= \big|F_{r_0}^{\pm}(\lambda,\mu;\psi) - F_{r}^{\pm}(\lambda,\mu;\psi) \big|\no\\
&\quad \geq \big|\big| F_{r}^{\pm}(\lambda,\mu;\psi) \big| - \big| F_{r_0}^{\pm}(\lambda,\mu;\psi) \big|\big|,\quad \lambda\in \bbR,\, \mu\in (0,\infty),\, r\in (0,r_0),\no
\end{align}
which yields
\begin{align}
&\big|F_{r}^{\pm}(\lambda,\mu;\psi)\big| 
\leq \big|F_{r_0}^{\pm}(\lambda,\mu;\psi)\big|
+ C_4\|\psi\|_{\LoneNs} \int_r^{r_0} \, d\varepsilon \, \varepsilon^{-1/2}\big|F_{\varepsilon}^{\pm}(\lambda,\mu;\psi)\big|^{1/2},    \no \\
&\hspace*{3.7cm}\psi\in \LoneN,\; \lambda\in \bbR,\; \mu\in (0,\infty), 
\;  r\in (0,r_0).    
\end{align}

By Lemma \ref{lA.1} with $\theta=1/2$,
\begin{align}
\big| F_r^{\pm}(\lambda,\mu;\psi) \big| \leq \big[ \big|F_{r_0}^{\pm}(\lambda,\mu;\psi)\big|^{1/2} + C_4\|\psi\|_{\LoneNs}(r_0^{1/2}-r^{1/2}) \big]^2,&\\
\psi\in \LoneN,\; \lambda\in \bbR,\; \mu\in (0,\infty),\; r\in (0,r_0).&\no
\end{align}
In addition, by Propositions \ref{p3.6} and \ref{p3.13}, 
\begin{align}
\big| F_{r_0}^{\pm}(\lambda,\mu;\psi)\big|^{1/2} 
&= |(\psi,G_{r_0}^{\pm}(\lambda,\mu)\psi)_{\LNs}|^{1/2}   \no \\
&\leq \big\|B^{-1/2}\psi\big\|_{\LNs}^{1/2} \big\|B^{1/2}
G_{r_0}^{\pm}(\lambda,\mu)\psi\big\|_{\LNs}^{1/2}   \\
&\leq C_1 r_0^{-1/2} \|\psi\|_{\LonehNs}, \quad \psi\in \LonehN,\; \lambda\in \bbR,\; \mu\in (0,\infty).\no
\end{align}
Therefore,
\begin{align}
\begin{split} 
\big| F_r^{\pm}(\lambda,\mu;\psi) \big| \leq \big[C_1 r_0^{-1/2}\|\psi\|_{\LonehNs} 
+ C_4\|\psi\|_{\LoneNs}(r_0^{1/2}-r^{1/2}) \big]^2,&\\
\psi\in \LoneN,\;\lambda\in \bbR,\; \mu\in (0,\infty),\; r\in (0,r_0).&\
\end{split} 
\end{align}
Finally, taking the limit $r\to 0^+$ and applying Proposition \ref{p3.10}, one obtains
\begin{equation}
\big| F_0^{\pm}(\lambda,\mu;\psi) \big| \leq \big[C_1 r_0^{-1/2} 
+ C_4r_0^{1/2} \big]^2\|\psi\|_{\LoneNs}^2,
\end{equation}
which yields the desired inequality with $C = \big[C_1 r_0^{-1/2} + C_4r_0^{1/2} \big]^2$.

The claim \eqref{2.76} is clear from the fact that $\langle Q \rangle^{-s}$ is an isometric 
isomorphism from $L^2_t (\bbR^n)$ onto $L_{s+t}^2 (\bbR^n)$, $s, t \in \bbR$. That 
$\langle Q \rangle^{-1}$ is $H_0$-Kato-smooth follows from \eqref{2.76} 
(cf., e.g., \cite[p.~148]{RS78}, \cite[p.~134]{Ya92}). 
\end{proof}

Of course, Theorem \ref{t3.15} implies the absence of any singular spectrum (i.e., the absence of eigenvalues and singular continuous spectrum) of $H_0$, 
\begin{equation}
\sigma_{s}(H_0) = \emptyset,
\end{equation}
but since $H_0$ is unitarily equivalent, via Fourier transform, to the operator 
of multiplication by $ \alpha \cdot p$, $p \in \bbR^n$, purely absolutely continuous spectrum of $H_0$ was obvious from the outset.  

\begin{remark} \lb{r3.16}
In the massive case, where $H_0(m) = H_0 + m \, \beta$, $m > 0$, a global limiting absorption principle for $H_0(m)$ was proved in dimension $n=3$ by Iftimovici and M{\u a}ntoiu 
\cite{IM99} in 1999. The corresponding massless case (i.e., for the operator $H_0$) in dimension $n=3$ was settled in 2008 by Sait{\= o} and Umeda \cite{SU08} employing entirely different methods in their study of zero eigenvalues and zero-energy resonances of massless Dirac operators. Our result, Theorem \ref{t3.15}, appears to be new for $n \in \bbN \backslash \{3\}$, $n \geq 2$.
${}$ \hfill $\diamond$
\end{remark}

Theorem \ref{t3.15} implies existence and completeness (in fact, unitarity) of wave operators for the pair of self-adjoint operators ($H=H_0 + V, H_0)$ for sufficiently ``weak'' perturbations $V$ of $H_0$ 
in the following standard manner (we refer to \cite[Ch.~4]{Ya92}, especially, 
\cite[Theorem~4.6.1]{Ya92}, for details): Consider the self-adjoint matrix-valued potential $V = \{V_{\ell,m}\}_{1 \leq \ell,m \leq N}$ satisfying for some $C \in (0,\infty)$, 
\begin{equation}
V \in \big[L^{\infty} (\bbR^n)\big]^{N \times N}, \quad 
\|V(x)\|_{\bbC^{N\times N}} \leq C \langle x \rangle^{- 2}\, \text{ for a.e. $x \in \bbR^n$},   \lb{3.92}
\end{equation}
and define the interacting massless Dirac operator $H$ via 
\begin{equation}
H = H_0 + V, \quad \dom(H) = \dom(H_0) = \WoneN.  \lb{3.93}
\end{equation}
One infers that
\begin{equation}
\|\langle x \rangle V \langle x\rangle\|_{\bbC^{N\times N}} \leq C\, \text{ for a.e. $x\in \bbR^n$},
\end{equation}
and 
\begin{align} \lb{3.105zz}
\begin{split} 
& \langle Q \rangle^{- 1} \big(H -z I_{\LNs} \big)^{-1} \langle Q \rangle^{- 1}  
= \langle Q \rangle^{- 1} \big(H_0 -z I_{\LNs} \big)^{-1} \langle Q \rangle^{- 1}   \\
& \quad \times \big[I + \langle Q \rangle V \, 
\big(H_0 -z I_{\LNs} \big)^{-1} \langle Q \rangle^{- 1} \big]^{-1},  
\quad z \in \bbC \backslash \bbR.  
\end{split} 
\end{align}

\begin{theorem} \lb{t3.17} 
Assume Hypotheses \ref{h3.1} and \ref{h3.2}, \eqref{3.92}, and define the operator $H$ as in \eqref{3.93}.  
If in addition, 
\begin{equation}
\sup_{\lambda \in \bbR,\, \mu \in (0,\infty)} \big\|\langle Q \rangle V 
(H_0 - (\lambda \pm i \mu) I_{\LNs})^{-1} \langle Q \rangle^{-1} \big\|_{\cB(\LNs)} < 1,
\end{equation}
then the wave operators
\begin{equation}
W_{\pm} (H,H_0) = \slim_{t \to \pm \infty} e^{i t H} e^{- i t H_0},  \quad 
W_{\pm} (H_0,H) = \slim_{t \to \pm \infty} e^{i t H_0} e^{- i t H}, 
\end{equation}
exist and are complete, that is, 
\begin{align} 
\begin{split} 
& \ker(W_{\pm} (H,H_0)) = \ker(W_{\pm} (H_0,H)) = \{0\}, \\
& \ran(W_{\pm} (H,H_0)) = \ran(W_{\pm} (H_0,H)) = \LN. 
\end{split} 
\end{align}
In fact, they are unitary and adjoint to each other, $W_{\pm} (H,H_0)^* = W_{\pm} (H_0,H)$. In 
particular, $H$ and $H_0$ are unitarily equivalent and hence $H$ is spectrally purely absolutely 
continuous. 
\end{theorem}

\begin{remark} \lb{r3.18}
Since we permit a (sufficiently decaying) matrix-valued potential $V$ in $H$, this includes, in particular, the case of electromagnetic interactions introduced via minimal coupling, that is, $V$ 
describes also special cases of the form,
\begin{equation}
H(q,A) := \alpha \cdot (-i \nabla - A) + q I_N = H_0 + [q I_N - \alpha \cdot A], \quad 
\dom(H(q,A)) = \WoneN,
\end{equation}
where $(q,A)$ represent the electromagnetic potentials on $\bbR^n$, with 
$q: \bbR^n \to \bbR$, $q \in L^{\infty}(\bbR^n)$, $A = (A_1,\dots,A_n)$, $A_j: \bbR^n \to \bbR$, 
$A_j \in L^{\infty}(\bbR^n)$, $ 1 \leq j \leq n$, and for some $C \in (0,\infty)$, 
\begin{equation}
|q(x)| + |A_j(x)| \leq C \langle x \rangle^{- 2}\, \text{ for a.e. $x \in \bbR^n$}, \; 1 \leq j \leq n.  
 \lb{3.101}
\end{equation}
${}$ \hfill $\diamond$
\end{remark}

\begin{remark} \lb{r3.19}
Using the notion of (local) strong operator smoothness as described in detail in 
\cite[Sects.~4.4--4.7]{Ya10}, the decay rate $\langle \, \cdot \, \rangle^{-2}$ in \eqref{3.92}, 
\eqref{3.101} can be relaxed to $\langle \, \cdot \, \rangle^{- \rho}$ for some $\rho > 1$ 
(cf.\ \cite{CGLNSZ17}) and this then permits situations where $H$ has eigenvalues and hence $H$ is no longer unitarily equivalent to $H_0$ and spectrally purely absolutely continuous. 
${}$ \hfill $\diamond$
\end{remark}

For additional (and more general) references in the context of smooth operator theory, limiting absorption principles, completeness of wave operators, and absence of singular continuous 
spectra, see, for instance, \cite{Ag75}, \cite{ABG96}, \cite{BH92}, \cite{BG10}, \cite{BMP93}, 
\cite{BM97}, \cite[Ch.~17]{BW83}, \cite{BD87}, \cite{Da05}, \cite{EGS08}, \cite{EGS09}, \cite{GM01}, \cite{Ge08}, \cite{GJ07}, \cite{Ka66}, \cite{Ku78}, \cite{MP96}, \cite{PSU95}, \cite{PSU98}, 
\cite[Sect.~XIII.7]{RS78}, \cite{Ri06}, \cite{RS12}, \cite{SU08}, \cite{Vo88}, 
\cite[Ch.~4]{Ya92}, \cite[Chs.~0--2]{Ya10}, \cite{Ya72}--\cite{Ya93}, in particular, global limiting absorption 
principles for Schr\"odinger operators can be found in \cite{EGS08}, \cite{EGS09}, \cite{RS12}. 

\medskip

\noindent
{\bf Acknowledgments.} 
We are indebted to Andrei Iftimovici for helpful discussions and for valuable hints to the literature on the commutator calculus.  We would also like to thank Eduard Tsekanovskii for helpful discussions. 

The authors are indebted to the Banff International Research Station for Mathematical Innovation and Discovery (BIRS) for their extraordinary hospitality during the focused research group on {\it Extensions of index theory inspired by scattering theory} (17frg668), June 18--25, 2017, where part of this work was done.

 
\end{document}